\documentclass[11pt]{amsart}

\usepackage{amssymb,amsmath} \usepackage{amsthm}
\usepackage{amsfonts,newlfont} \usepackage[dvips]{graphicx}

\usepackage{enumerate} \usepackage{euscript}

\newcommand{\T}{\mathbb{T}} \newcommand{\R}{\mathbb{R}}
\newcommand{\Z}{\mathbb{Z}}

 \newcommand{\Zk}{\mathbb Z^k}

  \renewcommand{\epsilon}{\varepsilon}

\newcommand{\supp}{\operatorname{supp}}

\newtheorem{theorem}{Theorem}[section]

\newtheorem{proposition}[theorem]{Proposition}
\newtheorem{corollary}[theorem]{Corollary}

\newtheorem{lemma}[theorem]{Lemma}

\theoremstyle{definition}

 \numberwithin{equation}{section} \begin{document}
\title[Rigidity of real-analytic actions  of $SL(n,\Z)$]{Rigidity of real-analytic actions  of $SL(n,\Z)$ on 
$\T^n$: A case of realization of Zimmer program}

\author[Anatole Katok and Federico Rodriguez Hertz]{Anatole
Katok *) and Federico Rodriguez Hertz **)} \address{The
Pennsylvania State University, University Park, PA}

\email{katok\_a@math.psu.edu} \address{IMERL, Montevideo,
Uruguay} \email{frhertz@fing.edu.uy} \date{\today} \thanks {*)
Based on research supported by NSF grant
DMS-0803880} \thanks{**) Partially supported by the Center for
Dynamics and Geometry at Penn State.} \maketitle

\begin{abstract}  We prove that  any real-analytic action of $SL(n,\Z),\,n\ge 3$
with  standard homotopy data that preserves  an ergodic   measure $\mu$ whose support is not contained in a ball, is analytically conjugate  on an open invariant set  to the standard  linear action on the complement to a finite union of periodic orbits.
\end{abstract} \bigskip

\section{Introduction. Formulation of results}
Let $G$ be a semisimple Lie group whose simple factors all have real rank  greater than one, or an 
irreducible lattice in such a group. ``Zimmer program''  first formulated in 1986  \cite{Z}  and modified in  the 1990s
to take into account examples described in \cite{KL},  aims at proving that  volume preserving 
actions of $G$  by diffeomorphisms of compact manifold $M$ are ``essentially algebraic''. This 
means that  $M$ splits into disjoint union
of open $G$-invariant  subsets $U_1,\dots, U_n$ and a nowhere  dense closed subset $F$ such 
that the restriction of the action to each of the open sets is smoothly conjugate to  the restriction of a 
certain standard algebraic action (homogeneous or affine) to an open dense $G$-invariant subset. 

Without  attempting an overview of results in the direction of Zimmer program let us point out that so 
far they have been either negative (such as non-existence of actions in   low dimension), or 
perturbative  (local differentiable rigidity of algebraic actions), or  subject to dynamical restrictions 
such as existence of Anosov elements. 

In   this note   we present  what is,  to the best of our knowledge, the first positive result free of such 
restrictions. 
Its principal limitation is that we consider real-analytic, rather  than  differentiable actions with the 
pay-off that the conjugacy is also real-analytic. 
On the other hand, instead of preservation of volume we make a weaker assumption of existence of 
an invariant measure with a ``homotopically large'' 
support. 

Let us consider the torus  $\T^n$. The group $SL(n,\Z)$ acts on $\T^n$ by automorphisms  which 
are projections of the linear maps in $\R^n$.  We will call  this  action of $SL(n,\Z)$  on $\T^n$ {\em 
the standard action} and denote it $\rho_0$. We will call the 
corresponding action on $\Z^n$ {\em the standard homotopy data}.  Since  $SL(n,\Z)$ is a lattice in 
the simple connected  Lie group $SL(n,\R)$  of real rank $n-1$ we assume $n\ge 3$.  Let $\rho$  
be an action of  a group $\Gamma\subset SL(n,\Z)$  by diffeomorphisms of $\T^n$. Induced action 
$\rho_*$ on $\Z^n=\pi_1(\T^n)$ is called {\em the homotopy data} of $\rho$. 
We will say that $\rho$ has {\em standard homotopy data} if $\rho_*$  is  the 
restriction of the standard homotopy data to $\Gamma$.

This note is dedicated to the proof of the following result:

\begin{theorem}\label{thm:main} Let $\Gamma\subset SL(n,\Z),\,n\ge 3$  be a finite index subgroup.
Let $\rho$ be a $C^\omega$ (real-analytic)  action of $\Gamma$ on $\T^n$ with standard homotopy data, 
preserving  an ergodic   measure $\mu$ whose support is not contained in a ball. Then:
\begin{enumerate}
\item There is a finite index subgroup $\Gamma'\subset \Gamma$, a  finite $\rho_0$-invariant  set $F$ and a bijective    real-analytic   map  $$H: \T^n
\setminus F\to  D$$ where $D$ is a dense subset of $\supp\mu$, such that for every $\gamma\in \Gamma'$,
$$H\circ\rho(\gamma)=\rho_0(\gamma)\circ H.$$  
\smallskip

\item
The map 
$H^{-1}$ can  be extended  to a  continuous (not necessarily invertible)  map $P: 
\T^n\to \T^n$ such that $\rho\circ P=\rho_0\circ P$. 
Moreover, for any $x\in F$,   pre-image  $P^{-1}(x)$  is  a connected set.
\smallskip
 
\item For $\Gamma=SL(n, \R)$ one can take $\Gamma'=\Gamma=SL(n, \R)$.
\end{enumerate}
\end{theorem}

Since Lebesgue measure $\lambda$ is the only non-atomic invariant measure  for  $\rho_0$ we 
deduce that $\mu=H_*\lambda$ and  we  have  the following
corollary

\begin{corollary} The measure $\mu$ is   given  by a real-analytic density on an open dense subset 
of its support.  Furthermore, it is the only ergodic $\rho$-invariant measure whose support is not 
contained in a ball. 
\end{corollary}

 The map $H$ is the inverse of  the  conjugacy  between  the action
$\rho$  on an open set  and the   standard action  in the 
complement  to  finitely many  periodic orbits. 

Thus $\rho$  is obtained by ``blowing up'' finitely many periodic orbits of $\rho_0$ and leaving the 
rest the same  up to a real-analytic time change. 
Possibility of such non-trivial  blow-ups on the torus  is an open question.  Constructions from \cite
{KL} produce  real-analytic actions  with blowups on some manifolds other than the torus, e.g by 
gluing in a projective space $\R P(n-1)$  by a $\sigma$-process
or glueing two  such projective spaces     and  thus attaching  a kind of handle to the torus. The 
same construction produces a  real-analytic action  on the torus  with an open round  hole that 
obviously cannot be extended to a real-analytic action inside the hole. $C^0$  extension is 
possible but  whether it can be extended  in a smooth way is an open question.

\section{Proof of Theorem \ref{thm:main}}
We will assume throughout this section that $\Gamma$  is a finite index  subgroup of $SL(n,\Z),\, n\ge 3$.

A particular case of \cite[Theorem, 6.10]{MQ} asserts that a $\Gamma$ action  on $\T^n$ with 
standard homotopy data preserving a  measure $\mu$  with full support is essentially semi-conjugate to the standard linear action. The following  proposition  is an improvement of that  
statement in two respects: (i)   we put a weaker condition on $\mu$ and (ii) we assert that $\mu$ is 
absolutely continuous. To achieve that we  rely on  results from \cite{KRH}  about  actions of Cartan 
(maximal rank  semisimple abelian)  subgroups of $SL(n,\Z)$. Let $\Gamma$ be a finite index 
subgroup of $SL(n,\Z)$. 
  
\begin{proposition}\label{prop:main}
Let $\rho$ be a $C^{1+\alpha}, \,\alpha>0$  $\Gamma$ action on $\T^n$ with standard  homotopy 
data. Assume that $\rho$ preserves  an ergodic measure $\mu$ whose support is not contained in 
a ball. Then $\mu$ is absolutely continuous  and  there is   a finite index subgroup  $\Gamma'$  of $\Gamma$ and a continuous map $h:\T^n \to \T^n$ such 
that  for every $\gamma\in \Gamma'$. 
\begin{equation}
\label{eq-conj}h\circ\rho(\gamma)=
\rho_0(\gamma)\circ h.\end{equation}
 and 
 \begin{equation}\label{eq-large} h_*\mu=\lambda.
 \end{equation} 
 Furthermore, $\mu$ is  unique measure satisfying  \eqref{eq-large}.   
\end{proposition}
\begin{proof}  
The scheme of the proof is as follows:
First we construct a  measurable  map $h$  defined $\mu$ almost everywhere and satisfying \eqref
{eq-conj}. Then  we show that $h$  extends  to a  continuous map defined on $\supp\mu$. Next we 
show that the   image  of this map is a complement to a finite set  and  \eqref{eq-large} holds. 
Finally  we  show that $h$ extends uniquely from $\supp \mu$   to the whole torus.  

Once we get this we argue as follows. Let $\nu=h_*\mu$; the measure $\nu$ is invariant and ergodic with respect to the  
linear $\Gamma'$ action, hence it is either Lebesgue or atomic. It will not be atomic 
(see bellow) hence $\nu$ is Lebesgue measure $\lambda$, this means that $\mu$ is a large 
measure  as defined in \cite{KRH}, i.e  $h_*\mu=\lambda$. Since for any Cartan subgroup $C
\subset SL(n,\R)$ large invariant measure is unique and absolutely continuous by \cite{KRH}, if we 
take a Cartan subgroup $C\subset\Gamma'$ (which can be done since $\Gamma'$ is of finite 
index in $SL(n,\Z)$), measure $\mu$ must coincide with this large measure and absolute continuity 
and uniqueness  follow. 

\smallskip

\noindent{\em Step 1.}
We shall use the following consequence of Zimmer's cocycle super-rigidity 
(see the proof of the Main Theorem in \cite{FW}):  
\begin{lemma}
Under the assumptions of Proposition~\ref{prop:main}  there is a 
measurable map $\phi:\T^n\to\R^n$ defined $\mu$ a.e. such that if we put $h_0(x)=x+\phi(x)$ then 
(\ref{eq-conj}) holds for $h_0$, i.e. 
$h_0\circ\rho(\gamma)=\rho_0(\gamma)\circ h_0$ for every $\gamma\in\Gamma'$ for some finite 
index subgroup $\Gamma'\subset\Gamma$.  
\end{lemma}
Let  $A$ be a set of full $\mu$ measure where the equality (\ref{eq-conj}) holds for $h_0$. Without 
loss of generality we may assume that $A$ is invariant with respect to  $\rho(\gamma)$ for every $\gamma\in
\Gamma'$. \smallskip

\noindent{\em Step 2.}  Now  we prove  that $\phi$ extends to a continuous map $\supp \mu \to \R^n$.

Let $\gamma\in\Gamma'$ be a hyperbolic matrix so that $\rho_0(\gamma)$ is an Anosov linear 
map. Then there is a continuous map $\phi_
\gamma: \T^n\to \R^n$ such that for $h_\gamma(x)=x+\phi_\gamma(x)$ one has $h_{\gamma}
\circ\rho(\gamma)=\rho_0(\gamma)\circ h_{\gamma}$. Let $v(x)=\phi_\gamma-\phi$.  $v\circ\rho
(\gamma)=\gamma v$ on $\T^n$, i.e. there is an integer $c\in\Z^n$ such that $v\circ\rho(\gamma)=
\gamma v+c$.  
This implies that $v(x)$ is  constant $\mu$ a.e. Indeed, let us write $c=(I-\gamma)c'$ and put $v'=v-
c'$, then  $$v'\circ\rho(\gamma)=v\circ\rho(\gamma)-c'=\gamma v+c-c'=\gamma v-\gamma c'=
\gamma v'.$$ Hence it is enough to see that $v'$ is constant $\mu$-a.e. Let $L_C$ be the set 
where $|v'(x)|<C$, taking $C$ large $L_C$ has measure as close to $1$ as wanted. Call $f=\rho(\gamma)$ and take $x\in L_C$ such that $x$ returns infinitely many 
times to $L_C$ in the future and the past. This set has full $\mu$ measure in $L_C$ by Poincar\'e 
recurrence. Then $v'(f^n(x))=\gamma^nv'(x)$. Observe that $\R^n$ splits as $E^u_\gamma\oplus 
E^s_\gamma$. Let $v'(x)$ decompose as $v'^s(x)+v'^u(x)$. Then $$C\geq |v'(f^n(x))|=|\gamma^n v'
(x)|\geq K|\gamma^n v'^u(x)|\geq K\lambda^n|v'^u(x)|$$ for some $\lambda>1$ which implies that 
$v'^u(x)=0$. Reversing time we obtain  $v'^s(x)=0$.  

Hence $\phi=\phi_\gamma-c'$ $\mu$ a.e. and $\phi$ extends continuously to the support of $\mu$; 
we will still denote this extended map by  
$\phi$. Let  $h:\supp(\mu)\to \T^n$, $h(x)=x+\phi(x)$, then we get that  $h\circ\rho(\gamma)=\rho_0
(\gamma)\circ h$ for every $\gamma\in\Gamma'$ since $h=h_0$ on a set 
of full $\mu$ measure. So, let us forget about $h_0$ and work with $h$. Notice  that for every 
hyperbolic $\gamma\in\Gamma'$ there is  $h_\gamma$ homotopic to the identity such that $h_
\gamma\circ \rho(\gamma)=\rho_0(\gamma)\circ h_\gamma$ and arguing as above we get that 
$h=h_\gamma+c'$ for some $c'\in(I-\Gamma)^{-1}(\Z^n)$ on the support of $\mu$. Observe also 
that changing $h_\gamma$ with $h_\gamma+c'$ we get that $h_\gamma+c'$ also conjugates $\rho
(\gamma)$ with $\gamma$. Hence we may assume already that $h$ and $h_\gamma$ coincide 
on $\supp \mu$ for every $\gamma\in\Gamma'$ hyperbolic. Since hyperbolic elements generate a finite index subgroup of  $\Gamma'$  we obtained desired map $h: \supp \mu \to \T^n$.
\smallskip

\noindent{\em Step 3.}  Let $\nu=h_*\mu$  be the push-forward of  the measure $\mu$. Obviously $
\nu$ is invariant and ergodic with respect to $\rho_0$ restricted to $\Gamma'$.   The only ergodic  
invariant measures for the $\Gamma'$ linear action $\rho_0$ are  Lebesgue  measure and 
measures 
supported on periodic orbits. Let us show that $\nu$ cannot have finite support. If this were the case, 
the support of $\mu$  would  belong to the union of pre-images of finitely many points. 
Take an element $\gamma\in\Gamma'$ inside a Cartan subgroup. Then we have 
that $h=h_\gamma$  and hence the pre-images of $h$ are inside the pre-images by $h_\gamma$. 
Now, the results in \cite{KRH} gives us that  the pre-image of every point by $h_\gamma$ is inside a 
ball, indeed the 
pre-image of every point is the intersection of  nested cubes. Hence the support of the measure is 
inside a finite disjoint union of cubes. Now, it is not hard to see that this finite disjoint union of cubes 
fit inside a ball, which implies that the support of $\mu$ is inside a ball, a contradiction. 

Thus we get that $\nu$ is Lebesgue measure and hence for every Anosov element of $\rho_0$, $
\mu$ is a large measure as defined in \cite{KRH}. 
\smallskip

\noindent{\em Step 4.}  Let us see that $h$ can be extended to the rest of the torus conjugating the 
whole action. Let us 
make the following observations. Let $R$ be the set of regular points which we may assume that 
are regular for all hyperbolic elements of the action. It has full $\mu$ measure. Then for every $x$ 
in $R$ 
and for every $\gamma\in\Gamma$ belonging to a Cartan subgroup we have that $W^s_\gamma
(x)\subset\supp(\mu)$. 

Let $U$ be a connected component of the complement of $\supp(\mu)$. Let us see that for every $
\gamma\in \Gamma$ an element in a Cartan subgroup $C_\gamma$, $h_\gamma(U)$ is a point 
and moreover this point does not depend on $\gamma$. Fix first $\gamma$ and take a point $x\in 
U$. We know from \cite{KRH} that  $x$ lies inside the intersection of a family of nested cubes $C_n
$ such that $\bigcap C_n=(h_\gamma)^{-1}h_\gamma(x)$ and $h_\gamma(C_n)$ is a family of 
cubes such that $\bigcap h_\gamma(C_n)=h_\gamma(x)$. The boundary of the cubes $C_n$ are 
formed by pieces of stable manifolds of regular elements with respect to different elements of $C_
\gamma$. Hence the boundary of $C_n$ is in $\supp(\mu)$ and since $U$ is connected and $U
\cap\supp(\mu)=\emptyset$ we get that $U\subset C_n$. In particular $h_\gamma(U)=h_\gamma(x)
$. So the whole $U$ is collapsed  by $h_\gamma$ into a single point. Let us take now another 
Cartan element $\gamma'$. Then $h_{\gamma'}(U)=h_{\gamma'}(x)$ also. Now, we want to proof 
that $h_\gamma(x)=h_{\gamma'}(x)$. But the boundary of $U$ is in the support of $\mu$ and on the 
support of $\mu$ we know that $h_\gamma$, $h_{\gamma'}$ and $h$ coincide. Hence $$h_
\gamma(x)=h_\gamma(U)=h_\gamma(\partial U)=h(\partial U)=h_{\gamma'}(\partial U)=h_
{\gamma'}(U)=h_{\gamma'}(x).$$
Since  Cartan subgroups generate $SL(n,\Z)$ this finishes the proof.
\end{proof}

The next  step in the proof of Theorem \ref{thm:main}  is finding a periodic orbit for  the action $\rho
$. 

\begin{proposition}\label{prop:properpoint}
Let $\rho$ be a $\Gamma$ action preserving a large invariant measure $\mu$ as in Proposition \ref
{prop:main}. Then if $\Gamma'$ is the finite index subgroup provided by Proposition \ref{prop:main}, 
there is a periodic point $p\in\supp\mu$ for the action $\rho$ restricted to $\Gamma'$ such that the derivative of $\rho$ at $p$ 
 coincides with the  standard linear action on $\R^n$.
\end{proposition}

\begin{proof} Let us recall another  result from \cite{KRH}.
\begin{theorem}\label{properperiodic}
Given an  action of $\Z^{n-1}$  on $\T^n$ with Cartan homotopy data, there is a proper periodic 
point, that is, if $h_\eta$ is a semiconjugacy with the linear action (work with a finite index subgroup 
if needed) then there is a periodic point $p$ for the action inside the support of the large 
measure such that $(h_\eta)^{-1}(h_\eta(p))=p$. Moreover the Lyapunov exponents for this point 
coincide with the Lyapunov exponents of the linear map. 
\end{theorem}

Recall that by the proof of Proposition \ref{prop:main} Step 2, we may assume that the 
semiconjugacy $h$ coincides with $h_\eta$ for every $\eta\in\Gamma'$. Hence, if $p$ is a proper 
periodic point for a Cartan action as in Theorem \ref{properperiodic} then $p=h^{-1}(h(p))$. Call 
$q=h(p)$ and observe that since $q$ is periodic for a linear Anosov map it has to be a rational 
point, hence it is periodic for the whole linear $SL(n,\Z)$ action. Finally, since $h\circ\rho(\gamma)=
\rho_0(\gamma)\circ h$ for every $\gamma\in\Gamma'$ we have that $p$ is periodic for the $\rho$ 
action restricted to $\Gamma'$. So, we may take a finite index subgroup $\Gamma''\subset
\Gamma'$ and assume that $p$ is fixed. We will assume that $\Gamma'$ already equals $
\Gamma''$. Since the derivative cocycle (homomorphism) at $p$ is nontrivial by Theorem \ref
{properperiodic}, Margulis super-rigidity Theorem implies that the derivative at $p$ coincides with the   
linear one up to some conjugacy that we may assume to be trivial by taking an appropriate coordinate 
chart at $p$. \end{proof} 

The next step is a  linearization of our action in a neighborhood  of the  periodic point $p$. This is the  only 
place where we use real analyticity  of the action.  We shall use the local 
linearization theorem of G. Cairns and E. Ghys for real analytic actions, \cite{CG}.  

\begin{theorem}\label{locallinearization}
Let $\Gamma$ be any irreducible lattice in a connected semi-simple Lie group with finite center, no 
non-trivial compact factor group and of rank greater  than 1. Every $C^\omega$-action of $\Gamma
$ on $(\R^m,0)$ is linearizable.
\end{theorem}

In particular, Theorem \ref{locallinearization} applies to $\Gamma'$, a finite index subgroup of $SL
(n,\Z)$.  In this case the linear action is conjugate to the standard action. So this theorem gives us a real analytic map $H:U\to \T^n$ where $U\subset\T^n$ is a 
neighborhood of $q$ such that $H\circ\rho_0(\gamma)(x)=\rho(\gamma)\circ H(x)$ for every $
\gamma\in\Gamma'$ and 
for every $x$ where the above equality make sense. $H(q)=p$ and we may assume without loss of 
generality that $D_qH=Id$.

To finish the proof of Theorem~\ref{thm:main} it is sufficient to show  that the semiconjugacy $h$ 
coincides with $H^{-1}$ and hence it is analytic and invertible  in a neighborhood $V$ of $p$. For, 
then  the set  of injectivity  for $h$  contains a $\rho$-invariant open set $\mathcal V=\bigcup_
{\gamma\in\Gamma'}\rho(\gamma)V$. The set  $h(\mathcal V)$ is open and $\rho_0$-invariant, 
hence is complement is finite.  But   $h$ is  real analytic on $\mathcal V$ since 
on $\rho(\gamma)V$ it coincides with  $\rho_0(\gamma)\circ H^{-1}\circ \rho(\gamma)$. 



\;\;

Let us see that $H^{-1}$ coincides with $h$ in a neighborhood of $p$. Let $\gamma\in\Gamma'$ be 
an element in a Cartan subgroup $C_\gamma$ and assume its stable manifold for the linear 
element is one dimensional. Call $W^s_{\gamma}(p)$ the stable manifold of $p$ 
for $\rho(\gamma)$ and $E^s_\gamma$ the stable space for $\rho_0(\gamma)$.  $C_\gamma$ 
acts locally 
transitively on $W^s_\gamma(p)$ and on $q+E^s_\gamma$. By uniqueness of the invariant 
manifolds, we have that $$H(q+E^s_\gamma\cap B_\epsilon(q))\subset W^s_\gamma(p).$$ We 
know from the results in \cite{KRH} that $h=h_\gamma$ is smooth at $W^s_\gamma(p)$. Hence we 
have that $H^{-1}$ coincides with $h$ along $W^s_\gamma(p)\cap B_\delta(p)$, where $\delta$ is 
such that $H(B_\epsilon(q))\supset B_\delta(p)$. Now, $A=\bigcup_\gamma q+E^s_\gamma\cap 
B_\epsilon(q)$ where $\gamma\in\Gamma'$ ranges over the elements belonging to a 
Cartan subgroup and with one-dimensional stable manifold is dense in $B_\epsilon(q)$. This is 
because the projective action on $S^{n-1}$ is minimal and $A$ corresponds with the orbit of a point 
in $S^{n-1}$, namely the orbit of the direction associated to the stable manifold of one of this $\gamma$'s. Hence $H(A)$ is dense in $B_\delta(p)$. On the other hand $h$ and $H^{-1}$ coincide on $H(A)$ hence by continuity they coincide on $B_\delta(p)$ which finishes the proof of statements (1) and (2)
of Theorem~\ref{thm:main}.    

To prove statement (3)  notice that $SL(n,\Z)$ is generated  by its maximal 
Cartan subgroups. For any such subgroup $C$ the semi-conjugacy $P$ 
conjugates restriction of $\rho$ to $C$ with an affine action $\rho_C$ with  standard homotopy data (see \cite{KKRH-errata} for a detailed proof). Elements of such an  action $\rho_C$ are compositions 
of automorphisms and rational translations.  In particular,  $\rho_C$ preserves
the finite set $Fix_C$ of fixed points of $\rho_0$ restricted to $C$. Since  the intersection of those fixed point sets for different Cartan subgroups is the identity, this implies  that the only affine action of $SL(n, \Z)$ with standard homotopy data is $\rho_0$. 

Thus  for any Cartan subgroup $C$, the action $\rho$ preserves the set 
$P^{-1}(Fix_C)$  and hence their intersection $P^{-1}(0)$.  This implies that
$\rho_C$  coincides with $\rho_0$ for any Cartan subgroup $C$ and statement (3) follows. 

\section{Remarks and open problems}
With a  more careful analysis of the structure  of finite index subgroups of 
 $SL(n, \Z)$ one can very likely prove that 
the semiconjugacy $P$ also  serves as   a semiconjugacy  between the whole action $\rho$ and an   affine  action $\tilde\rho_0$ with the standard homotopy data.  Basically the question reduces  to consideration of finite order elements 
of $\Gamma$ that  are not products of elements of   Cartan subgroups.

A much more interesting question  concerns  restrictions of $\rho_0$ to  infinite index subgroups  whose linear representations are  rigid, such as integer
lattices in  other higher rank simple Lie groups or irreducible representations
of $SL(n,\Z)$  into $SL(N, \Z)$ for large $N$. While Cairns-Ghys Theorem~\ref{locallinearization} is available in those settings, the other key ingredient, 
 a weaker form of rigidity for maximal abelian subgroups  like in \cite{kk, KRH}, is missing. 
Progress  in this direction    beyond  the Cartan case has been achieved  recently in \cite{KRH2} but the strong simplicity condition  of that paper is not satisfied  in the interesting cases  mentioned above. Furthermore, in some cases, such as the representation of $SL(n,\Z)$  into $SL(n^2-1, \Z)$
given by the  adjoint action on traceless matrices, there are no Anosov elements altogether and one  should hope to tie elements of rigidity for  partially hyperbolic  abelian subgroups together to produce rigidity for the whole action. 

A more immediate and  probably more accessible issue is  extension of  our results to the differentiable case. Here the situation is reversed: rigidity  for Cartan actions is   proven but local linearization  is not, and it may not even be 
true. Rigidity  for Cartan actions provides  extensive information about the
semi-conjugacy $P$. Already for  a Cartan subgroup, it is smooth in  Whitney sense  on  sets whose measure  is arbitrary close to full measure, see \cite{KRH2}. Those 
sets include grids of codimension one smooth submanifolds that  divide the  space into ``boxes'', most of them small. Superimposing those pictures for different Cartan subgroups  provides such grids in a dense set of  directions. 
However, certain elements of uniformity  needed to   conclude smoothness on an open set, are  lacking.

\bibliographystyle{alpha}  \end{document}